\documentclass[12pt]{article} 
\usepackage{amsfonts,amsmath,latexsym,amssymb,mathrsfs,amsthm}
\usepackage{slashbox}
\usepackage{caption}
\usepackage[colorlinks=true]{hyperref}

\let\OLDthebibliography\thebibliography
\renewcommand\thebibliography[1]{
  \OLDthebibliography{#1}
  \setlength{\parskip}{3pt}
  \setlength{\itemsep}{0pt plus 0.3ex}
}

\newcommand{\ga}{\alpha}

\newcommand{\gG}{\Gamma}

\newcommand{\rarr}{\rightarrow}

\def\numberlikeadb{\global\def\theequation{\thesection.\arabic{equation}}}
\numberlikeadb
\newtheorem{theorem}{Theorem}[section]
\newtheorem{lemma}[theorem]{Lemma}

\newtheorem{proposition}[theorem]{Proposition}
\newtheorem{remark}[theorem]{Remark}

\usepackage{lscape}
\usepackage{caption}
\usepackage{multirow}
\begin{document}

\title{An asymptotic expansion for the normalizing constant of the Conway-Maxwell-Poisson distribution}
\author{Robert E. Gaunt\footnote{School of Mathematics, The University of Manchester, Manchester M13 9PL, UK},\, Satish Iyengar\footnote{Department of Statistics, University of Pittsburgh, Pittsburgh, PA 15260, USA}, \, \\
 Adri B. Olde Daalhuis\footnote{Maxwell Institute and School of Mathematics, The University of Edinburgh,  EH9 3FD, UK}  \, and Burcin Simse$\mathrm{k}^\dagger$
}

\date{} 
\maketitle

\vspace{-10mm}

\begin{abstract}
The Conway-Maxwell-Poisson distribution is a two-parameter generalisation of the Poisson distribution that can be used to model data that is under- or over-dispersed relative to the Poisson distribution. The normalizing constant $Z(\lambda,\nu)$ is given by an infinite series that in general has no closed form, although several papers have derived approximations for this sum.  In this work, we start by using probabilistic argument to obtain the leading term in the asymptotic expansion of $Z(\lambda,\nu)$ in the limit $\lambda\rightarrow\infty$ that holds for all $\nu>0$.  We then use an integral representation to obtain the entire asymptotic series and give explicit formulas for the first eight coefficients. We apply this asymptotic series to obtain approximations for the mean, variance, cumulants, skweness, excess kurtosis and raw moments of CMP random variables.  Numerical results confirm that these correction terms yield more accurate estimates than those obtained using just the leading order term.
\end{abstract}


\noindent{{{\bf{AMS 2010 Subject Classification:}}} Primary 60E05; 62E20; 41A60; 33C20. 

\section{Introduction}\label{intro}

The Conway-Maxwell-Poisson (CMP) distribution (also known as the COM-Poisson distribution) is a natural two-parameter generalisation of the Poisson distribution that was introduced by Conway and Maxwell \cite{cm62} as the stationary number of occupants of a queuing system with state-dependent service or arrival rates.  The first in-depth studies of the CMP distribution were carried out by Boatwright et al$.$ \cite{bbk03} and Shmueli et al$.$ \cite{s05}.  Since then the distribution has received attention in the statistics literature on account of the flexibility it offers in statistical models.  In particular, the CMP distribution is useful for modelling data that is under- or over-dispersed relative to the Poisson distribution.  Sellers and Shmueli \cite{ss10} have used the CMP distribution to generalise the Poisson and logistic regression models; Kadane et al$.$ \cite{k06} considered the use of the CMP distribution in Bayesian analysis; the CMP distribution is also employed in a flexible cure rate model formulated by Rodrigues et al$.$ \cite{r09}; and a survey of further applications of the CMP distribution is given in Sellers et al$.$ \cite{sbs12}.

We shall write $X\sim\mbox{CMP}(\lambda,\nu)$ if
\begin{equation}\label{admpdf}
\mathbb{P}(X=j)=\frac{1}{Z(\lambda,\nu)}\frac{\lambda^j}{(j!)^\nu},\quad j=0,1,2,\ldots,
\end{equation}   
where $Z(\lambda,\nu)$ is a normalizing constant defined by
$$
Z(\lambda,\nu)=\sum_{i=0}^\infty\frac{\lambda^i}{(i!)^\nu}.
$$
The domain of admissible parameters for which (\ref{admpdf}) defines a probability distribution is $\lambda,\nu>0$, and $0<\lambda<1$, $\nu=0$.  Its distributional properties were first studied by Shmueli et al$.$ \cite{s05}, and a comprehensive account is given in a recent work of Daly and Gaunt \cite{dg16}.  The focus of this paper, however, is the normalizing constant $Z(\lambda,\nu)$.  Many important summary statistics of the CMP distribution can be expressed in terms of $Z(\lambda,\nu)$ (see \cite{s05}, \cite{dg16} and Section 3), which motivates studying the properties of $Z(\lambda,\nu)$.

Plainly, $X\sim\mbox{CMP}(\lambda,1)$ has the Poisson distribution $\mbox{Po}(\lambda)$ and the normalizing constant $Z(\lambda,1)=e^\lambda$.  
As noted by Shmueli et al$.$ \cite{s05}, other choices of $\nu$ give rise to well-known distributions.  Indeed, if $\nu=0$ and $0<\lambda<1$, then $X$ has a 
geometric distribution, with $Z(\lambda,0)=(1-\lambda)^{-1}$.  In the limit $\nu\rightarrow\infty$, $X$ converges in distribution to a Bernoulli random variable with mean $\lambda(1+\lambda)^{-1}$ 
and $\lim_{\nu\rightarrow\infty}Z(\lambda,\nu)=1+\lambda$.  It was noted by \c{S}im\c{s}ek and Iyengar \cite{bs-si} and Daly and Gaunt \cite{dg16} that $Z(\lambda,2)=I_0(2\sqrt{\lambda})$, 
where $I_0(x)=\sum_{k=0}^\infty\frac{1}{(k!)^2}\big(\frac{x}{2}\big)^{2k}$ is a modified Bessel function of the first kind.  Finally, Nadarajah \cite{n09} noted that, for integer $\nu$, the normalizing 
constant can expressed as generalized hypergeometric function: $Z(\lambda,\nu)=$\! $_0F_{\nu-1}(;1,\ldots,1;\lambda)$ (see Appendix \ref{appendixC} for a definition).  
An elegant integral representation of the normalizing constant has also recently been obtained by Pog\'{a}ny \cite{tibor}.

In general, however, the normalizing constant $Z(\lambda,\nu)$ does not permit a closed-form expression in terms of elementary functions.   Asymptotic results are available, however.  Shmueli et al$.$ \cite{s05} proved that, for fixed positive integer $\nu$,
\begin{equation}\label{eq:norm}
Z(\lambda,\nu)=\frac{\exp\left\{\nu\lambda^{1/\nu}\right\}}{\lambda^{(\nu-1)/2\nu}(2\pi)^{(\nu-1)/2}\sqrt{\nu}}\left(1+\mathcal{O}\left(\lambda^{-1/\nu}\right)\right), \quad \text{as $\lambda\rightarrow\infty$.}
\end{equation}
Their approach involved a Laplace approximation of a $(\nu-1)$-dimensional integral representation of $Z(\lambda,\nu)$, and does not generalise to non-integer $\nu$.  However, they conjectured that (\ref{eq:norm}) was valid for all $\nu>0$.  Earlier, Olver \cite{o74} had used contour integration to derive the leading term in the asymptotic expansion of $Z(\lambda,\nu)$ for $0<\nu\leq4$.  Gillispie and Green \cite{gg14} built on the work of \cite{o74} to confirm that (\ref{eq:norm}) holds for all $\nu>0$.  

In a recent work, \c{S}im\c{s}ek and Iyengar \cite{bs-si} used a considerably simpler probabilistic argument to derive the approximation (\ref{eq:norm}).  Their approach involved expressing the normalizing constant as an expectation involving a Poisson random variable, which is then approximated by a normal random variable to yield an approximation for $Z(\lambda,\nu)$.  However, as was noted by the authors, a slight gap in their argument meant that their derivation was not rigorous.  In this paper, we fill in this gap, which results in an elegant and rigorous derivation of (\ref{eq:norm}) that holds for all $\nu>0$.  

The natural next question is to ask for lower order terms in the asymptotic expansion of $Z(\lambda,\nu)$.  In Theorem \ref{thm3.1}, we obtain the entire asymptotic expansion for $\nu>0$, and give explicit formulas for the first eight terms in the expansion.   
 We arrive at our asymptotic series by recalling that $Z(\lambda,\nu)=$\! $_0F_{\nu-1}(;1,\ldots,1;\lambda)$ for integer $\nu$. 
This leads us to a new integral representation which is just a
generalization of the integral representation \href{http://dlmf.nist.gov/16.5.E1}{16.5.1} in \cite{NIST:DLMF}
\begin{equation}\label{mainintrep}
Z(\lambda,\nu)=\frac{1}{2\pi i}\int_L\frac{\Gamma(-t)\left(-\lambda\right)^t}{\left(\Gamma(t+1)\right)^{\nu-1}}\, dt,
\end{equation}
where the contour $L$ starts at infinity on a line parallel to the positive real axis, encircles the nonnegative integers in the negative sense, and ends at infinity on another line parallel to the positive real axis.
The integral converges for $\nu>0$ and $\lambda\not=0$. Hence, $\nu$ does not have to be a positive integer. The methods in Lin and Wong \cite{lw16} can be used to obtain the complete asymptotic expansion for this
integral. The details are in Appendix \ref{appendixC}. 

\begin{theorem}\label{thm3.1}Fix $\nu>0$.  Then
\begin{equation}\label{thmex}Z(\lambda,\nu)=\frac{\exp\left\{\nu\lambda^{1/\nu}\right\}}{\lambda^{(\nu-1)/2\nu}(2\pi)^{(\nu-1)/2}\sqrt{\nu}}\sum_{k=0}^\infty c_k\big(\nu\lambda^{1/\nu}\big)^{-k}, \quad \text{as $\lambda\rightarrow\infty$,}
\end{equation}
where the $c_j$ are uniquely determined by the expansion
\begin{equation}\label{cjcj}
\left(\Gamma(t+1)\right)^{-\nu}=\frac{\nu^{\nu (t+1/2)}}{\left(2\pi\right)^{(\nu-1)/2}}\sum_{j=0}^\infty\frac{c_j}{\Gamma(\nu t+(1+\nu)/2+j)}.
\end{equation}
In particular, $c_0=1$, $c_1=\frac{\nu^2-1}{24}$, $c_2=\frac{\nu^2-1}{1152}\left(\nu^2+23\right)$.  For more coefficients see (\ref{cccdef}).
\end{theorem}

This is the first instance in which correct lower order terms have been obtained, and numerical results, given in Section \ref{numerics}, show that these correction terms yield more accurate estimates than those in the existing literature.  The lower order terms are given in terms of the $c_j$, which can be determined from (\ref{cjcj}).  The $c_j$ can be readily obtained using computer algebra, and in Appendix \ref{appendixC} we provide a simple computational method and give the first eight terms.  Our code is available at {\tt http://www.maths.ed.ac.uk/$\sim$adri/CMPdistr.mw.zip}.  

For any fixed $\nu>0$, the asymptotic expansion is valid as $\lambda\rightarrow\infty$.  However, because the asymptotic expansion is given in terms of negative powers of $\lambda^{1/\nu}$, for sufficently large $\lambda$, the approximation will be particularly accurate for small $\nu$ (the over-dispersed $\nu<1$ case), although the approximation is accurate even in the under-dispersed $\nu>1$ case.  Indeed, as seen in Section \ref{numerics}, even in the fairly moderate $\lambda=1.5$ case, taking the first three terms in (\ref{thmex}) gives an absolute error of less than $1\%$ for $\nu=1.9$, and the approximation is improved further for smaller values of $\nu$.  Based on the numerical results of Section \ref{numerics}, we consider that a safe rule of thumb for obtaining accurate approximations using the asymptotic approximation (\ref{thmex}), with the first three terms, is for both $\lambda\geq 1.5$ and $\lambda^{1/\nu}\geq 1.5$ to hold (the absolute error was always less than $0.5\%$ in our tests), but we refer the reader to Tables \ref{table1} and \ref{table2} for numerical results that provide further insight into the quality of the approximation. 

As mentioned above, certain important summary statistics of the CMP distribution can be expressed in terms of the normalizing constant $Z(\lambda,\nu)$, and our asymptotic series (\ref{thmex}) allows us to obtain more accurate estimates for these summaries than those in the current literature.  In Section \ref{applications}, we demonstrate this by applying expansion (\ref{thmex}) to obtain approximations for the mean, variance, cumulants, skweness, excess kurtosis and raw moments of CMP random variables.

The rest of this article is organised as follows.  In Section \ref{rigor-lead}, we fill in the gap in the derivation of \c{S}im\c{s}ek and Iyengar \cite{bs-si} to obtain a rigorous probabilistic proof of the leading term in the asymptotic expansion of $Z(\lambda,\nu)$.  In Section \ref{applications}, we use Theorem \ref{thm3.1} to derive approximations for several summary statistics of the CMP distribution.  In Section \ref{numerics}, we present numerical results that confirm that the lower order correction terms result in more accurate estimates of $Z(\lambda,\nu)$.  We discuss our results in Section \ref{discuss}. Finally, in Appendix \ref{appendixaaa}, we prove Theorem \ref{thm3.1} and the technical Lemma \ref{lem1}.

\section{A probabilistic derivation of the leading term in the expansion}\label{rigor-lead}

In this section, we give a probabilistic derivation of the asymptotic formula (\ref{eq:norm}).  The basic approach is the same as the one given in \c{S}im\c{s}ek and Iyengar \cite{bs-si}, but additional care is taken to make the proof rigorous.  After presenting our proof, we shall compare our argument to theirs; see Remark \ref{rem1}.  Before giving the proof we state a technical lemma, which we prove in Appendix \ref{appendixA}.

\begin{lemma}\label{lem1}  Let $X_\alpha\sim\mathrm{Po}(\alpha)$, where $\alpha>0$, and set $\tilde{X_\alpha}=\frac{X_\alpha-\alpha}{\sqrt{\alpha}}$.  Let $Z\sim N(0,1)$ have the standard normal distribution.  Suppose that $h:\mathbb{R}\rightarrow\mathbb{R}$ does not depend on $\alpha$.    

(i) If $h$ is differentiable with bounded derivative on $\mathbb{R}$, then
\begin{equation*}\mathbb{E}[h(\tilde{X_\alpha})]=\mathbb{E}[h(Z)]+\mathcal{O}(\alpha^{-1/2}), \quad \mbox{as } \alpha\rightarrow\infty.
\end{equation*}

(ii) If $h$ is an even function ($h(x)=h(-x)$ for all $x\in\mathbb{R}$), and $h$ is twice differentiable with first and second derivative bounded on $\mathbb{R}$, then 
\begin{equation*}\mathbb{E}[h(\tilde{X_\alpha})]=\mathbb{E}[h(Z)]+\mathcal{O}(\alpha^{-1}), \quad \mbox{as } \alpha\rightarrow\infty.
\end{equation*}

\end{lemma} 

\noindent{\bf{Proof of (\ref{eq:norm}) for all $\nu>0$.}}  For convenience let us reparametrize the CMP family with $\alpha = \lambda^{1/\nu}$.  Then the asymptotic formula (\ref{eq:norm}) for $C(\alpha,\nu):=Z(\lambda^{1/\nu},\nu)$ becomes
\begin{equation}\label{normalizingterm}C(\alpha,\nu)= \frac{e^{\nu \alpha}}{(2 \pi \alpha)^{\frac{\nu-1}{2}} \sqrt{\nu}} \left[1+ \mathcal{O}(\alpha^{-1})\right],\quad \mbox{as } \alpha\rightarrow\infty.
\end{equation}
The proof involves three steps.  First, we write the normalising constant $C(\alpha,\nu)$ as an expectation of a function of the random variable $X_\alpha\sim\mathrm{Po}(\alpha)$.  We then obtain an asymptotic approximation for this function, and finally use a normal approximation of the Poisson variate $X_\alpha$ to approximate the expectations. \\

\noindent{\bf{Step 1. Express normalizing constant as an expectation.}} Notice that
\begin{equation}
\label{diffrepresentation}
C(\ga,\nu) = \sum_{k=0}^{\infty} \left(\frac{\alpha^{k}}{k!} \right)^{\nu} = e^{\alpha} \sum_{k=0}^{\infty} e^{-\alpha} \frac{\alpha^{k}}{k!} \left(\frac{\alpha^{k}}{k!} \right)^{\nu-1} = e^{\alpha} \mathbb{E}\left[\left( \frac{\alpha^{X_\alpha}}{X_\alpha!} \right)^{\nu-1}\right],
\end{equation}
where $X_\alpha\sim \mathrm{Po}(\alpha)$.  When $\alpha$ is large, $X_\alpha$ will, with high probability, be concentrated around $\alpha$.  Recalling Stirling's approximation $ \Gamma(x+1) \approx x^{x} e^{-x} \sqrt{2 \pi x}$, as $x \rarr \infty$,  motivates writing the normalizing constant as
\begin{equation}
\label{newnormalizingterm}
C(\alpha,\nu)  =  e^{\alpha} \mathbb{E} \left(\frac{\alpha^{X_\alpha}}{\alpha^{\alpha} e^{-\alpha} \sqrt{2 \pi \alpha}} \frac{\alpha^{\alpha} e^{-\alpha} \sqrt{2 \pi \alpha}}{\Gamma(X_\alpha+1)} \right)^{\nu-1} = \frac{e^{\nu \alpha}}{(2 \pi \alpha)^{\frac{\nu-1}{2}}} \mathbb{E}[f_\alpha(X_\alpha)],
\end{equation}
where
\begin{equation*}f_\alpha(x)=\bigg(\frac{\alpha^xe^{-\alpha}\sqrt{2\pi\alpha}}{\Gamma(x+1)}\bigg)^{\nu-1}.
\end{equation*}
Note that the constant term here is the same as the asymptotic expression in (\ref{normalizingterm}) above, except for the $\sqrt{\nu}$ term in the denominator.  The next steps of the proof involve showing that $\mathbb{E}[f_\alpha(X_\alpha)]= \nu^{-1/2}(1+\mathcal{O}(\alpha^{-1}))$, as $\alpha\rightarrow\infty$. \\

\noindent{\bf{Step 2. Approximation of $f_\alpha$.}} The central limit theorem says that as $\alpha\rightarrow\infty$, $\frac{X_\alpha-\alpha}{\sqrt{\alpha}}\stackrel{\mathcal{D}}{\rightarrow}Z$, where $Z\sim N(0,1)$ (see Appendix \ref{appendixA} for further details).  This motivates reparametrizing $f_\alpha$ as follows:
\begin{equation*}g_\alpha(x)=f_\alpha(\alpha+x\sqrt{\alpha})=\bigg(\frac{\alpha^{\alpha+x\sqrt{\alpha}}e^{-\alpha}\sqrt{2\pi\alpha}}{\Gamma(\alpha+x\sqrt{\alpha}+1)}\bigg)^{\nu-1}.
\end{equation*}
We shall now obtain an asymptotic approximation for $g_\alpha$ and later apply it together the with approximation $X_\alpha\stackrel{\mathcal{D}}{\approx}\alpha+Z\sqrt{\alpha}$.  We begin by writing
\begin{equation}\label{lngeqn}\frac{\ln(g_\alpha(x))}{\nu-1}=\bigg(\alpha+x\sqrt{\alpha}+\frac{1}{2}\bigg)\ln(\alpha)-\alpha+\frac{\ln(2\pi)}{2}-\ln\Gamma(\alpha +x\sqrt{\alpha}+1).
\end{equation}
Let us now note two useful asymptotic expansions:
\begin{eqnarray}
\label{asy1}
 \ln \gG(t+1) &=& \left(t+\frac{1}{2}\right) \ln(t+1) - (t+1) + \frac{\ln(2\pi)}{2} + \mathcal{O}(t^{-1}), \quad \mbox{for } t \rarr \infty, 
 \\ 
\label{asy2} \ln(1+t) &=& t - \frac{t^{2}}{2} +\frac{t^3}{3}+  \mathcal{O}(t^{4}), \quad \mbox{for } t\rightarrow0.  
 \end{eqnarray}
From (\ref{asy1}) we have that
\begin{align*}&\ln\Gamma(\alpha+x\sqrt{\alpha}+1)\\
&= \bigg(\alpha+x\sqrt{\alpha}+\frac{1}{2}\bigg)\ln(\alpha+x\sqrt{\alpha}+1)-(\alpha+x\sqrt{\alpha}+1)+\frac{\ln(2\pi)}{2}+\mathcal{O}(\alpha^{-1}) \\
&= \bigg(\alpha+x\sqrt{\alpha}+\frac{1}{2}\bigg)\bigg[\ln(\alpha)+\ln\bigg(1+\frac{x}{\sqrt{\alpha}}+\frac{1}{\alpha}\bigg)\bigg]-(\alpha+x\sqrt{\alpha}+1)+\frac{\ln(2\pi)}{2}+\mathcal{O}(\alpha^{-1}).
\end{align*}
Substituting into (\ref{lngeqn}) and then applying the asymptotic formula (\ref{asy2}) gives
\begin{align*}\frac{\ln(g_\alpha(x))}{\nu-1}&= x\sqrt{\alpha}+1- \bigg(\alpha+x\sqrt{\alpha}+\frac{1}{2}\bigg)\ln\bigg(1+\frac{x}{\sqrt{\alpha}}+\frac{1}{\alpha}\bigg)+\mathcal{O}(\alpha^{-1}) \\
&= x\sqrt{\alpha}+1- \bigg(\alpha+x\sqrt{\alpha}+\frac{1}{2}\bigg)\bigg[\frac{x}{\sqrt{\alpha}}+\frac{1}{\alpha}-\frac{1}{2}\bigg(\frac{x}{\sqrt{\alpha}}+\frac{1}{\alpha}\bigg)^2\\
&\quad+\frac{1}{3}\bigg(\frac{x}{\sqrt{\alpha}}+\frac{1}{\alpha}\bigg)^3\bigg]+\mathcal{O}(\alpha^{-1}) \\
&= -\frac{x^2}{2}+\frac{x^3-3x}{6\sqrt{\alpha}}+\mathcal{O}(\alpha^{-1}).
\end{align*}
Taking exponentials and using $e^{t} = 1+t+\mathcal{O}(t^2)$ as $t\rightarrow0$ gives 
\begin{align}
\label{Uterm}
g_\alpha(x)&=\exp\bigg(-(\nu-1)\frac{x^2}{2}\bigg)\exp\bigg((\nu-1)\frac{x^3-3x}{6\sqrt{\alpha}}+\mathcal{O}(\alpha^{-1})\bigg) \nonumber \\
&= \exp\bigg(-(\nu-1)\frac{x^2}{2}\bigg)\bigg[1+(\nu-1)\frac{x^3-3x}{6\sqrt{\alpha}}+\mathcal{O}(\alpha^{-1})\bigg], \quad \mbox{as } \alpha\rightarrow\infty.
\end{align} 

\noindent {\bf{Step 3. Approximation of expectations.}} From the approximation (\ref{Uterm}), we have 
\begin{align*}\mathbb{E}[f_\alpha(X_\alpha)]&=\mathbb{E}[g_\alpha(\tilde{X_\alpha})] \\
&= \mathbb{E}\Big[e^{-(\nu-1)\tilde{X_\alpha}^2/2}\Big]+\frac{\nu-1}{6\sqrt{\alpha}}\mathbb{E}\Big[\big(\tilde{X_\alpha}^3-3\tilde{X_\alpha}\big)e^{-(\nu-1)\tilde{X_\alpha}^2/2}\Big]+\mathcal{O}(\alpha^{-1}),
\end{align*}
where $\tilde{X_\alpha}\stackrel{\mathcal{D}}{=}\frac{X_\alpha-\alpha}{\sqrt{\alpha}}$.  We now use Lemma \ref{lem1} (see also Remark \ref{rem0}) to approximate the expectations involving the random variable $\tilde{X_\alpha}$ by the corresponding expectations of the standard normal variate $Z$ to obtain
\begin{align}\mathbb{E}[f_\alpha(X_\alpha)]&= \Big\{\mathbb{E}\Big[e^{-(\nu-1)Z^2/2}\Big]+\mathcal{O}(\alpha^{-1})\Big\}\nonumber \\
\label{for1}&\quad+\Big\{\frac{\nu-1}{6\sqrt{\alpha}}\Big(\mathbb{E}\Big[\big(Z^3-3Z\big)e^{-(\nu-1)Z^2/2}\Big]+\mathcal{O}(\alpha^{-1/2})\Big)\Big\}+\mathcal{O}(\alpha^{-1}) \\
\label{for2}&= \mathbb{E}\Big[e^{-(\nu-1)Z^2/2}\Big]+\frac{\nu-1}{6\sqrt{\alpha}}\mathbb{E}\Big[\big(Z^3-3Z\big)e^{-(\nu-1)Z^2/2}\Big]+\mathcal{O}(\alpha^{-1}). 
\end{align}
Here, in applying Lemma \ref{lem1}, we made use of the fact that $h_1(x)=e^{-(\nu-1)x^2/2}$ is an even function that is twice differentiable with bounded first and second derivative on $\mathbb{R}$, and that $h_1(x)=(x^3-3x)e^{-(\nu-1)x^2/2}$ has a bounded derivative on $\mathbb{R}$.  
The second expectation of (\ref{for2}) is equal to zero, because $(x^3-3x)e^{-(\nu-1)x^2/2}$ is an odd function.  We therefore have that
\begin{align*}\mathbb{E}[f_\alpha(X_\alpha)]= \int_{-\infty}^\infty\frac{1}{\sqrt{2\pi}}e^{-\nu t^2/2}\,dt+\mathcal{O}(\alpha^{-1})=\frac{1}{\sqrt{\nu}}+\mathcal{O}(\alpha^{-1}),
\end{align*}
so that
\begin{equation*}
 C(\alpha,\nu) = e^{\alpha} \mathbb{E}\left[ \left( \frac{\alpha^{X_\alpha}}{X_\alpha!} \right)^{\nu-1}\right] = \frac{e^{\nu \alpha}}{(2 \pi \alpha)^{\frac{\nu-1}{2}} \sqrt{\nu}}  [1 + \mathcal{O}(\alpha^{-1})], \quad \mbox{ as } \alpha \rarr \infty,
\end{equation*}
which completes the proof.  \hfill $\Box$

\begin{remark}\label{rem0}
Using Lemma \ref{lem1} allows us to quantify the size of the error in the approximation (\ref{for1}).  We could have derived the leading term in the asymptotic formula (\ref{eq:norm}) through a simpler argument by appealing to the fact that since $\tilde{X_\alpha}$ convergences in distribution to the standard normal distribution we have that $\mathbb{E}[h(\tilde{X_\alpha})]\rightarrow\mathbb{E}[h(Z)]$, as $\alpha\rightarrow\infty$, for all bounded functions $h:\mathbb{R}\rightarrow\mathbb{R}$.  However, this would have led to the weaker result that $C(\alpha,\nu)= \frac{e^{\nu \alpha}}{(2 \pi \alpha)^{\frac{\nu-1}{2}} \sqrt{\nu}} \left[1+ o(1)\right]$, as $\alpha\rightarrow\infty$.
\end{remark}

\begin{remark}\label{rem1} The basic outline of our proof follows that of \c{S}im\c{s}ek and Iyengar \cite{bs-si}, who used the approximation $2\sqrt{X_\alpha} \stackrel{\mathcal{D}}{\approx} N(2\sqrt{\alpha},1)$, or $X_\alpha\stackrel{\mathcal{D}}{\approx}\alpha+Z\sqrt{\alpha}+\frac{Z^2}{4}$. Instead, here we use the central limit approximation $X_\alpha\stackrel{\mathcal{D}}{\approx}\alpha+Z\sqrt{\alpha}$, which is simpler and leads to no error terms of larger order. Most importantly, our approach allows us to appeal to results for approximating expectations from the Stein's method literature.  Applying these results allows us to make the derivation rigorous and quantify the size of the error in the approximation.
\end{remark}


\section{Applications: approximation of summary statistics}\label{applications}

Many important summary statistics, such as moments and cumulants, of the CMP distribution can be expressed in terms of the normalizing constant $Z(\lambda,\nu)$; see Section 2 of Daly and Gaunt \cite{dg16} for a comprehensive account, where all formulas below can be found.  Let $X\sim \mathrm{CMP}(\lambda,\nu)$.  The probability generating function is $\mathbb{E}s^X=\frac{Z(s\lambda,\nu)}{Z(\lambda,\nu)}$, and the mean and variance are given by
\begin{align}\label{easym}\mathbb{E}X&=\lambda\frac{d}{d\lambda}\big\{\ln(Z(\lambda,\nu))\big\}, \\
\label{varasym}\mathrm{Var}(X)&=\lambda\frac{d}{d\lambda}\mathbb{E}X.
\end{align}
The cumulant generating function is
$$
g(t)=\ln(\mathbb{E}[e^{tX}])=\ln(Z(\lambda e^{t},\nu))-\ln(Z(\lambda,\nu)),
$$
and the cumulants are given by
\begin{equation}\label{cueqn}
\kappa_n=g^{(n)}(0)=\frac{\partial^n}{\partial t^n}\ln(Z(\lambda e^{t},\nu))\bigg|_{t=0}, \quad n\geq1.
\end{equation}
The skewness $\gamma_1=\frac{\kappa_3}{\sigma^3}$ and excess kurtosis $\gamma_2=\frac{\kappa_4}{\sigma^4}$, where $\sigma^2=\mathrm{Var}(X)$, can therefore also be expressed in terms of $Z(\lambda,\nu)$.  As moments can be expressed in terms of cumulants, it follows that they in turn can be expressed in terms of $Z(\lambda,\nu)$.  Let $\mu_n'=\mathbb{E}X^n$.  Then
\begin{equation}\label{bell}\mu_n'=\sum_{k=1}^nB_{n,k}(\kappa_1,\ldots,\kappa_{n-k+1}),
\end{equation}
where the partial Bell polynomial (see Hazelwinkel \cite{h97}, p$.$ 96) is given by
\begin{equation*}B_{n,k}(x_1,x_2,\ldots,x_{n-k+1})=\sum\frac{n!}{j_1!j_2!\cdots j_{n-k+1}!}\bigg(\frac{x_1}{1!}\bigg)^{j_1}\bigg(\frac{x_2}{2!}\bigg)^{j_2}\cdots\bigg(\frac{x_{n-k+1}}{(n-k+1)!}\bigg)^{j_{n-k+1}},
\end{equation*}
where the sum is taken over all sequences $j_1, j_2, j_3,\ldots, j_{n-k+1}$ of non-negative integers such that the following two conditions hold:
\begin{align*}&j_1+j_2+\cdots+j_{n-k+1}=k, \\
&j_1+2j_2+3j_3+\cdots+(n-k+1)j_{n-k+1}=n.
\end{align*}
The central moments $\mu_n=\mathbb{E}[(X-\mathbb{E}X)^n]$ can be obtained by setting $\kappa_1=0$ in (\ref{bell}).

For general parameter values, there do not exist simple closed form formulas for these summary statistics.  However, due to the asymptotic approximations of $Z(\lambda,\nu)$, we can obtain approximations.  Shmueli et al$.$ \cite{s05} used the formula (\ref{easym}) and the approximation (\ref{eq:norm}) for the normalizing constant to obtain the approximation
\begin{equation}\label{schex}\mathbb{E}X\approx \lambda^{1/\nu}-\frac{\nu-1}{2\nu}, \quad \text{as $\lambda\rightarrow\infty$.}
\end{equation}
Their result was given for integer $\nu$, but as the approximation (\ref{eq:norm}) has since been shown to be valid for all $\nu>0$, it follows that (\ref{schex}) also holds for all $\nu>0$.  Daly and Gaunt \cite{dg16} used the approximation (\ref{eq:norm}) to derive the leading order term in the asymptotic expansion of a number of further summary statistics, which hold for all $\nu>0$.  

In the following proposition, we use the asymptotic expansion (\ref{thmex}) to obtain additional correction terms for several important summary statistics.  
The expansions we present include the first four terms in the asymptotic series for the above summary statistics, 
except for the moments for which we obtain the first three terms owing to more complicated expressions.  With further calculations, one could readily obtain additional terms, 
although this would complicate the exposition and lead to only negligible improvements in accuracy.

\begin{proposition}Let $X\sim \mathrm{CMP}(\lambda,\nu)$, where $\nu>0$.  Then, as $\lambda\rightarrow\infty$,
\begin{align}\label{prop11}\mathbb{E}X&= \lambda^{1/\nu}\left(1-\frac{\nu-1}{2\nu} \lambda^{-1/\nu}-\frac{\nu^2-1}{24\nu^2}\lambda^{-2/\nu}-\frac{\nu^2-1}{24\nu^3}\lambda^{-3/\nu}+\mathcal{O}(\lambda^{-4/\nu}) \right),\\
\label{prop22}\mathrm{Var}(X)&= \frac{\lambda^{1/\nu}}{\nu}\bigg(1+\frac{\nu^2-1}{24\nu^2}\lambda^{-2/\nu}+\frac{\nu^2-1}{12\nu^3}\lambda^{-3/\nu}+\mathcal{O}(\lambda^{-4/\nu})\bigg),\\
\label{prop33}\kappa_n&= \frac{\lambda^{1/\nu}}{\nu^{n-1}}\bigg(1+\frac{(-1)^n(\nu^2-1)}{24\nu^2}\lambda^{-2/\nu}+\frac{(-2)^n(\nu^2-1)}{48\nu^3}\lambda^{-3/\nu}+\mathcal{O}(\lambda^{-4/\nu})\bigg), \\
\label{prop44}\gamma_1&= \frac{\lambda^{-1/2\nu}}{\sqrt{\nu}}\bigg(1-\frac{5(\nu^2-1)}{48\nu^2}\lambda^{-2/\nu}-\frac{7(\nu^2-1)}{24\nu^3}\lambda^{-3/\nu}+\mathcal{O}(\lambda^{-4/\nu})\bigg), \\
\label{prop55}\gamma_2&= \frac{\lambda^{-1/\nu}}{\nu}\bigg(1-\frac{(\nu^2-1)}{24\nu^2}\lambda^{-2/\nu}+\frac{(\nu^2-1)}{6\nu^3}\lambda^{-3/\nu}+\mathcal{O}(\lambda^{-4/\nu})\bigg), \\
\label{prop66}\mu_n'&= \lambda^{n/\nu}\bigg(1+\frac{n(n-\nu)}{2\nu}\lambda^{-1/\nu}+a_2\lambda^{-2/\nu}+\mathcal{O}(\lambda^{-3/\nu})\bigg),
\end{align}
where
\begin{equation*}a_2=-\frac{n(\nu-1)(6n\nu^2-3n\nu-15n+4\nu+10)}{24\nu^2}+\frac{1}{\nu^2}\bigg\{\binom{n}{3}+3\binom{n}{4}\bigg\}.
\end{equation*}
The asymptotic series (\ref{prop33}) for $\kappa_n$ holds for all $n\geq2$, and $\kappa_1=\mathbb{E}X$.
\end{proposition}

\begin{proof}We start with the mean.  To obtain the desired level of accuracy it suffices to truncate the asymptotic series (\ref{thmex}) for $Z(\lambda,\nu)$ at the third term:
\begin{equation*}Z(\lambda,\nu)=\frac{\exp\left\{\nu\lambda^{1/\nu}\right\}}{\lambda^{(\nu-1)/2\nu}(2\pi)^{(\nu-1)/2}\sqrt{\nu}}\big(1+c_1\big(\nu\lambda^{1/\nu}\big)^{-1}+
c_2\big(\nu\lambda^{1/\nu}\big)^{-2}+\mathcal{O}(\lambda^{-3/\nu})\big),
\end{equation*}
as $\lambda\rightarrow\infty$.  Taking logarithms gives  
\begin{align*}\ln(Z(\lambda,\nu))&=\nu\lambda^{1/\nu}-\frac{\nu-1}{2\nu}\ln(\lambda)+C_\nu+\ln\big(1+c_1\nu^{-1}\lambda^{-1/\nu}+c_2\nu^{-2}\lambda^{-2/\nu}+\mathcal{O}(\lambda^{-3/\nu})\big).
\end{align*}
where $C_\nu=-\ln\big((2\pi)^{(\nu-1)/2}\sqrt{\nu}\big)$.  Applying the asymptotic formula
\begin{equation*}\ln(1+ax+bx^2+\mathcal{O}(x^3))= ax+(b-a^2/2)x^2+\mathcal{O}(x^3), \quad \text{as $x\rightarrow0$,}
\end{equation*}
then yields
\begin{equation}\label{logeq}\ln(Z(\lambda,\nu))=\nu\lambda^{1/\nu}-\frac{\nu-1}{2\nu}\ln(\lambda)+C_\nu+\frac{c_1}{\nu}\lambda^{-1/\nu}+\frac{(c_2-c_1^2/2)}{\nu^2}\lambda^{-2/\nu}+\mathcal{O}(\lambda^{-3/\nu}).
\end{equation}
We now differentiate the asymptotic series to obtain
\begin{align*}\mathbb{E}X=\lambda\frac{d}{d\lambda}\big\{\ln(Z(\lambda,\nu))\big\}= \lambda^{1/\nu}-\frac{\nu-1}{2\nu}-\frac{c_1}{\nu^2}\lambda^{-1/\nu}-\frac{2}{\nu^3}(c_2-c_1^2/2)\lambda^{-2/\nu}+\mathcal{O}(\lambda^{-3/\nu}),
\end{align*}
which on recalling that $c_1=\frac{\nu^2-1}{24}$ and noting that $c_2-c_1^2/2=\frac{\nu^2-1}{48}$ yields the desired asymptotic series for the mean.  Here we differentiated the asymptotic series (\ref{logeq}) in the naive sense by simply differentiating term by term.  This is justifiable here (and throughout this proof) because the remainder term is of the form $\sum_{k=3}^\infty a_k \lambda^{-k/\nu}$. However, as noted by \cite{h91}, p$.$ 23, asymptotic series cannot be differentiated in this manner in general.  The asymptotic series (\ref{prop22}) for the variance now follows from differentiating the series (\ref{prop11}) and applying the formula (\ref{varasym}).

We now move on to the cumulants.  As $\kappa_1=\mathbb{E}X$, we restrict our attention to $n\geq2$.  From (\ref{logeq}) we have
\begin{align*}\ln(Z(\lambda e^t,\nu))&= \nu\lambda^{1/\nu}e^{t/\nu}-\frac{\nu-1}{2\nu}t+C_{\lambda,\nu}+c_1\big(\nu\lambda^{1/\nu}e^{t/\nu}\big)^{-1}\\
&\quad+(c_2-c_1^2/2)\big(\nu\lambda^{1/\nu}e^{t/\nu}\big)^{-2}+\mathcal{O}(\lambda^{-3/\nu}),
\end{align*}
where $C_{\lambda,\nu}=-\frac{\nu-1}{2\nu}\ln(\lambda)+C_\nu$.  Using (\ref{cueqn}) and differentiating gives, for $n\geq2$,
\begin{align*}\kappa_n= \frac{\partial^n}{\partial t^n}\ln(Z(\lambda e^{t},\nu))\bigg|_{t=0}= \frac{\lambda^{1/\nu}}{\nu^{n-1}}+\frac{(-1)^nc_1}{\nu^{n+1}}\lambda^{-1/\nu}+\frac{(-2)^n(c_2-c_1^2/2)}{\nu^{n+2}}\lambda^{-2/\nu}+\mathcal{O}(\lambda^{-3/\nu}),
\end{align*}
which on substituting $c_1=\frac{\nu^2-1}{24}$ and $c_2-c_1^2/2=\frac{\nu^2-1}{48}$ yields (\ref{prop33}).

We now obtain the asymptotic expansions for the skewness $\gamma_1=\frac{\kappa_3}{\sigma^3}$ and excess kurtosis $\gamma_2=\frac{\kappa_4}{\sigma^4}$.  From (\ref{prop33}), (\ref{prop44}) and (\ref{prop22}) we have
\begin{align}\gamma_1&=\frac{\lambda^{-1/2\nu}}{\sqrt{\nu}}\bigg(1-\frac{\nu^2-1}{24\nu^2}\lambda^{-2/\nu}-\frac{\nu^2-1}{6\nu^3}\lambda^{-3/\nu}+\mathcal{O}(\lambda^{-4/\nu})\bigg)\nonumber \\
\label{gam111}&\quad\times\bigg(1+\frac{\nu^2-1}{24\nu^2}\lambda^{-2/\nu}+\frac{\nu^2-1}{12\nu^3}\lambda^{-3/\nu}+\mathcal{O}(\lambda^{-4/\nu})\bigg)^{-3/2},
\end{align}
and
\begin{align}\gamma_2&=\frac{\lambda^{-1/\nu}}{\nu}\bigg(1+\frac{\nu^2-1}{24\nu^2}\lambda^{-2/\nu}+\frac{\nu^2-1}{3\nu^3}\lambda^{-3/\nu}+\mathcal{O}(\lambda^{-4/\nu})\bigg)\nonumber \\
\label{gam222}&\quad\times\bigg(1+\frac{\nu^2-1}{24\nu^2}\lambda^{-2/\nu}+\frac{\nu^2-1}{12\nu^3}\lambda^{-3/\nu}+\mathcal{O}(\lambda^{-4/\nu})\bigg)^{-2}.
\end{align}
Applying the asymptotic formula
\begin{equation*}\frac{1+ax^2+bx^3}{(1+cx^2+dx^3)^n}= 1+(a-nc)x^2+(b-nd)x^3+\mathcal{O}(x^4), \quad \text{as $x\rightarrow0$,}
\end{equation*}
with $n=3/2$ and $n=2$ to (\ref{gam111}) and (\ref{gam222}), respectively, then yields (\ref{prop44}) and (\ref{prop55}).  

Finally, we obtain the asymptotic expansions for the moments $\mu_n'$.  Note that $\kappa_1,\ldots,\kappa_n$ are all of order $\lambda^{1/\nu}$ as $\lambda\rightarrow\infty$.  Hence, from (\ref{bell}) and the asymptotic expansions (\ref{prop11}) for $\mathbb{E}X=\kappa_1$ and (\ref{prop33}) for the cumulants $\kappa_n$, we have
\begin{align*}\mu_n'&= \kappa_1^n+\binom{n}{2}\kappa_1^{n-2}\kappa_2+\binom{n}{3}\kappa_1^{n-3}\kappa_3+3\binom{n}{4}\kappa_1^{n-4}\kappa_2^2+\mathcal{O}(\lambda^{(n-3)/\nu}) \\
&=\lambda^{n/\nu}\bigg(1-\frac{\nu-1}{2\nu}\lambda^{-1/\nu}-\frac{\nu^2-1}{24\nu^2}\lambda^{-2/\nu}+\mathcal{O}(\lambda^{-3/\nu})\bigg)^n \\
&+\frac{\lambda^{(n-1)/\nu}}{\nu}\binom{n}{2}\bigg(1-\frac{\nu-1}{2\nu}\lambda^{-1/\nu}+\mathcal{O}(\lambda^{-2/\nu})\bigg)^{n-2}\Big(1+\mathcal{O}(\lambda^{-2/\nu})\Big) \\
&\quad+\frac{\lambda^{(n-2)/\nu}}{\nu^2}\bigg\{\binom{n}{3}+3\binom{n}{4}\bigg\}+\mathcal{O}(\lambda^{(n-3)/\nu}) \\
&= \lambda^{n/\nu}+\bigg\{-\frac{n(\nu-1)}{2\nu}+\frac{1}{\nu}\binom{n}{2}\bigg\}\lambda^{(n-1)/\nu}+\bigg\{\binom{n}{2}\bigg(\frac{\nu-1}{2\nu}\bigg)^2-\frac{n(\nu^2-1)}{24\nu^2}\\
&\quad-\frac{1}{\nu}\binom{n}{2}\frac{(n-2)(\nu-1)}{2\nu}+\frac{1}{\nu^2}\bigg\{\binom{n}{3}+3\binom{n}{4}\bigg\}\bigg\}\lambda^{(n-2)/\nu}+\mathcal{O}(\lambda^{(n-3)/\nu}),
\end{align*} 
which on simplifying yields (\ref{prop66}).  The proof is complete.
\end{proof}

\section{Numerical results} \label{numerics}

\begin{table}[h]
\begin{center}
\caption{\footnotesize{Percentage error for the approximation of $Z(\lambda,\nu)$ by the leading order term, with first order correction, and with second order correction in (\ref{thmex}).  A negative number means that the approximation is less than the true value.  Errors greater than $100\%$ are denoted by 101.}}
\label{table1}
{\scriptsize
\begin{tabular}{|c|rrrrrrrrrr|}
\hline
 \backslashbox{$\lambda$}{$\nu$}      &    0.1 &    0.3 &    0.5 &    0.7 &    0.9 &    1.1 &    1.3 &    1.5 &    1.7 &    1.9 \\
 \hline
0.1 &  $-100$  & $-78.7$ &  $-35.8$  & $-10.8$ &   $-1.39$  &   0.106  &  $-1.89$  &  $-5.24$  &  $-8.96$ &  $-12.6$ \\
  0.1 &  $-101$ &  $-101$ & $-101$ &  $-83.4$ &  $-12.6$  &   6.57  &  10.9  &  10.0  &   7.38   &  4.19 \\ 
     0.1 & $-101$ &  $-101$ &   $-101$ &  $-101$ &  $-92.3$ &   30.6 &   40.5 &   34.9  &  27.5  &  20.7 \\

\rule{0pt}{3ex}0.3 & $-97.7$ &  $-38.4$  &  $-6.81$  &   1.10  &   1.08 &   $-1.37$ &   $-4.42$ &   $-7.44$  & $-10.2$ &  $-12.7$ \\
  0.3 & $-101$ & $-101$ &  $-71.5$ &  $-16.0$ &   $-2.31$ &    0.974 &    0.918  &  $-0.268$ &   $-1.78$  &  $-3.31$ \\
  0.3 & $-101$ & $-101$ & $-101$ &  $-83.0$ &   $-9.42$ &    4.19 &    6.25 &    5.34 &    3.66  &   1.87 \\

\rule{0pt}{3ex}0.5 &  $-83.3$ &  $-12.0$ &    2.87 &    3.70  &   1.42 &   $-1.45$ &   $-4.24$ &   $-6.77$ &   $-9.02$ &  $-11.0$ \\
  0.5 & $-101$ & $-101$ &  $-22.9$ &   $-4.77$ &   $-0.505$ &    0.023  &  $-0.631$  &  $-1.64$  &  $-2.68$ &   $-3.66$ \\
  0.5 & $-101$ & $-101$ & $-101$ &  $-20.7$ &   $-2.80$ &    1.29 &    1.80 &    1.23 &    0.338  &  $-0.575$ \\ 
 
\rule{0pt}{3ex}0.7 & $-49.1$  &   2.16 &    6.05 &    4.16 &    1.38 &   $-1.33$ &   $-3.78$  &  $-5.96$ &   $-7.86$ &   $-9.54$ \\
  0.7 & $-101$ &  $-40.2$ &   $-7.47$ &   $-1.11$ &    0.058 &   $-0.244$ &   $-0.985$  &  $-1.81$ &   $-2.56$  &  $-3.29$ \\
  0.7 & $-101$ & $-101$ &  $-34.2$ &   $-7.23$ &   $-1.03$ &    0.444  &   0.473 &    0.030 &   $-0.538$ &   $-1.09$ \\ 
 
\rule{0pt}{3ex}0.9 &  $-8.59$ &    7.77 &    6.59 &    3.93  &   1.25 &  $-1.18$ &   $-3.34$ &   $-5.24$ &   $-6.91$  &  $-8.39$ \\
  0.9 & $-101$ &  $-11.6$ &   $-1.63$   &  0.258 &    0.248 &   $-0.315$ &   $-1.02$  &  $-1.71$  &  $-2.32$ &   $-2.85$ \\
  0.9 & $-101$ &   $-55.7$ &  $-11.5$ &   $-2.72$  &  $-0.372$  &   0.121  &  $-0.022$ &   $-0.377$ &   $-0.775$ &   $-1.14$ \\ 
 
\rule{0pt}{3ex}1.1 &   9.76  &   8.51  &   6.05 &    3.48 &    1.10 &   $-1.04$  &  $-2.94$  &  $-4.64$ &   $-6.15$  &  $-7.49$ \\ 
  1.1 &  $-7.70$  &  $-1.47$ &    0.575  &   0.740  &   0.300  &  $-0.317$ &   $-0.949$  &  $-1.53$ &   $-2.04$ &   $-2.46$ \\
  1.1 & $-40.0$ &  $-13.1$  &  $-3.81$  &  $-0.933$ &   $-0.097$  &  $-0.014$  &  $-0.215$ &   $-0.507$ &   $-0.802$  &  $-1.06$ \\ 
 
\rule{0pt}{3ex}1.3  &  4.80  &   7.08 &    5.15  &   3.00  &   0.960  &  $-0.915$ &   $-2.61$  &  $-4.14$ &   $-5.53$ &   $-6.77$ \\
  1.3  &  1.67  &   1.43  &   1.26 &    0.853  &   0.296 &   $-0.294$ &   $-0.852$  &  $-1.35$ &   $-1.78$ &   $-2.13$ \\
  1.3  &  0.576  &  $-2.34$ &   $-0.970$  &  $-0.180$ &    0.023  &  $-0.069$  &  $-0.283$  &  $-0.527$ &   $-0.755$ &   $-0.944$ \\ 
 
\rule{0pt}{3ex}1.5 &   0.813 &    5.09 &    4.20 &    2.56   &  0.835 &   $-0.809$ &   $-2.34$ &   $-3.74$   & $-5.02$ &   $-6.19$ \\
  1.5 &   0.092 &    1.65  &   1.31 &    0.812 &    0.270 &   $-0.263$ &   $-0.754$  &  $-1.19$ &   $-1.55$ &   $-1.85$ \\
  1.5 &   0.032 &    0.227  &   0.062 &    0.129   &  0.072 &   $-0.090$  &  $-0.296$  &  $-0.503$  &  $-0.686$ &   $-0.829$ \\  
 
\rule{0pt}{3ex}1.7 &   0.210  &   3.36 &    3.36  &   2.17 &    0.729 &   $-0.719$ &   $-2.10$  &  $-3.40$  &  $-4.60$ &   $-5.71$ \\
  1.7  &  0.005  &   1.14  &   1.12 &    0.713 &    0.238 &   $-0.231$ &   $-0.662$ &   $-1.04$  &  $-1.36$   & $-1.63$ \\
  1.7  &  0.000 &    0.526 &    0.371 &    0.237   &  0.087 &   $-0.093$ &   $-0.284$  &  $-0.462$ &   $-0.612$  &  $-0.724$ \\ 
 
\rule{0pt}{3ex}1.9  &  0.068 &    2.14 &    2.65  &   1.84  &   0.639 &   $-0.643$ &   $-1.91$ &   $-3.11$ &   $-4.25$  &  $-5.30$ \\
  1.9  &  0.001  &   0.620  &   0.875  &   0.600 &    0.205  &  $-0.202$  &  $-0.581$  &  $-0.918$ &   $-1.20$  &  $-1.44$ \\
  1.9  &  0.000 &    0.333  &   0.398  &   0.255   &  0.088 &   $-0.089$ &   $-0.262$ &   $-0.416$  &  $-0.543$ &   $-0.633$ \\
  \hline
\end{tabular}
}
\end{center}
\end{table}

\begin{table}[h]
\begin{center}
\caption{\footnotesize{Percentage error for the approximation of $Z(\lambda,\nu)$ by the leading order term, with first order correction, and with second order correction in (\ref{thmex}).  A negative number means that the approximation is less than the true value.}}
\label{table2}
{\scriptsize
\begin{tabular}{|c|rrrrrr|}
\hline
 \backslashbox{$\lambda$}{$\nu$}      &    2.5 &    3 &    3.5 &    4 &    4.5 &    5   \\
 \hline
3 & $-6.40$& $-8.24$ & $-9.88$ & $-11.3$ & $-12.6$ & $-13.6$  \\
3 & $-1.12$& $-1.17$ & $-1.06$ & $-0.791$ &  $-0.356$  &  0.257 \\ 
3 & $-0.292$& $-0.083$ & 0.291  & 0.834    &  1.56     &  2.48  \\ 

\rule{0pt}{3ex}4 & $-5.58$ & $-7.42$ & $-9.12$ & $-10.7$ &   $-12.1$ & $-13.4$ \\
  4 & $-0.838$ & $-0.936$ & $-0.930$ & $-0.822$ & $-0.604$ &   $-0.257$ \\
  4 & $-0.174$ & $-0.029$ & 0.226 &  0.595 & 1.09 & 1.73 \\

\rule{0pt}{3ex}5 & $-5.04$ & $-6.86$ & $-8.60$ & $-10.3$ &   $-11.8$ & $-13.3$ \\
  5 & $-0.679$ & $-0.809$ & $-0.875$ & $-0.881$ & $-0.821$ &   $-0.675$ \\
  5 & $-0.120$ & $-0.023$ & 0.149 &  0.393 & 0.718 & 1.15 \\

 \rule{0pt}{3ex}6 & $-4.65$ & $-6.45$ & $-8.21$ & $-9.92$ &   $-11.6$ & $-13.1$ \\
  6 & $-0.577$ & $-0.727$ & $-0.839$ & $-0.923$ & $-0.978$ &   $-0.984$ \\
  6 & $-0.092$ & $-0.027$ & 0.088 &  0.244 & 0.446 & 0.713 \\

 \rule{0pt}{3ex}7 & $-4.35$ & $-6.12$ & $-7.88$ & $-9.63$ &   $-11.3$ & $-13.0$ \\
  7 & $-0.505$ & $-0.666$ & $-0.808$ & $-0.945$ & $-1.08$ &   $-1.20$ \\
  7 & $-0.075$ & $-0.032$ & 0.044 &  0.139 & 0.250 & 0.395 \\

 \rule{0pt}{3ex}8 & $-4.10$ & $-5.85$ & $-7.61$ & $-9.37$ &   $-11.1$ & $-12.9$ \\
  8 & $-0.451$ & $-0.617$ & $-0.778$ & $-0.951$ & $-1.15$ &   $-1.35$ \\
  8 & $-0.064$ & $-0.036$ & 0.013 &  0.065 & 0.112 & 0.163 \\

\rule{0pt}{3ex}9 & $-3.90$ & $-5.62$ & $-7.37$ & $-9.14$ &   $-10.9$ & $-12.7$ \\
  9 & $-0.409$ & $-0.576$ & $-0.749$ & $-0.946$ & $-1.18$ &   $-1.46$ \\
  9 & $-0.056$ & $-0.038$ & 0.007 &  0.015 & 0.015 & $-0.005$ \\ 

\rule{0pt}{3ex}10 & $-3.73$ & $-5.42$ & $-7.16$ & $-8.93$ &   $-10.7$ & $-12.6$ \\
  10 & $-0.375$ & $-0.542$ & $-0.720$ & $-0.932$ & $-1.20$ &   $-1.52$ \\
  10 & $-0.049$ & $-0.039$ & $-0.021$ & $-0.018$ & $-0.052$ & $-0.126$ \\  
  \hline
\end{tabular}
}
\end{center}
\end{table}

Tables \ref{table1} and \ref{table2} give the percentage errors in approximating the normalizing $Z(\lambda,\nu)$ by the asymptotic series (\ref{thmex}).  The values of $\lambda$ and $\nu$ in Table \ref{table1} are the same as those used by Shumeli et al$.$ \cite{s05}, who also looked at the percentage errors of the leading term in the approximation.  The values of $\lambda$ and $\nu$ given in Table \ref{table2} provide additional insight into the quality of the approximation in a different parameter regime. Looking at Table \ref{table1}, we see that when $\lambda$ is small the approximation underestimates $Z(\lambda,\nu)$, and the approximation is particularly poor when $\nu$ is also small.  In these cases, including additional corrections terms actually leads to a worse approximation.  However, as expected, the approximation improves as $\lambda$ increases, particularly when $\lambda^{1/\nu}$ increases.  For $\lambda\geq1.1$, including the first correction term always leads to a more accurate approximation, and when $\lambda\geq1.5$ including both the first and second order correction always leads to the most accurate estimate.  For $\lambda=1.9$, including second order correction gives estimates that are more accurate by an order of magnitude than using just the leading order term, and the absolute error is always less than $1\%$.  One can also see from Table \ref{table1} that the approximation is particularly accurate when $\nu$ is close to 1.  This is to be expected because when $\nu=1$ the leading term in the asymptotic expansion (\ref{thmex}) reduces, for all $\lambda>0$, to $e^{\lambda}$, meaning that the approximation is exact since $Z(\lambda,1)=e^{\lambda}$.   

We carried out a similar analysis for the mean, variance, skewness and excess kurtosis.  We obtained very similar results, which is unsurprising given that all the these summary statistics can be expressed in terms of $Z(\lambda,\nu)$.  For space reasons, we omit the results.  

Lastly, we remark that the largest value of the normalizing constant in our study was $Z(1.9,0.1)=5.49743309747796\times 10^{28}$.  For larger $\lambda$, approximating $Z(\lambda,\nu)$ by truncating the exact sum at some large value would be computationally challenging, whereas the asymptotic approximation (\ref{thmex}) offers an accurate and computationally efficient alternative. Indeed, taking eight terms in the asymptotic approximation gives us
$5.49743309747884\times 10^{28}$, and hence the relative error is $1.59\times 10^{-13}$.

\section{Discussion}\label{discuss}

The problem of approximating the CMP normalizing constant $Z(\lambda,\nu)$ dates back to Shmueli et al$.$ \cite{s05}, who derived the leading order term in the asymptotic expansion for large $\lambda$ and integer $\nu$.  Their conjecture that the approximation is valid for all $\nu>0$ was recently confirmed by Gillispie and Green \cite{gg14}.  In this paper, we complemented this result by filling in a gap in a work of \c{S}im\c{s}ek and Iyengar \cite{bs-si} to obtain a simpler probabilistic derivation.

Whilst novel and interesting, these approaches do not easily allow one to obtain lower order terms.  The main contribution in this paper is Theorem \ref{thm3.1}, which gives the entire expansion for all $\nu>0$.  We arrived at our proof by recognising $Z(\lambda,\nu)$ as a generalized hypergeometric function, for which there is a large body of literature on its asymptotics, some quite recent.  We then simply generalised results of Lin and Wong \cite{lw16} to obtain the expansion.

The coefficients in the expansion are given in terms of quantities $c_j$ that are uniquely determined by (\ref{cjcj}).  We stated the first eight terms, which should suffice for most practical purposes. Further terms will involve computer algebra. We used our expansion to approximate several important summary statistics.  We gave the first four terms, which again we expect to suffice in most practical situations, and straightforward calculations would allow one to readily obtain further terms.

Part of the motivation for studying the asymptotics of the normalizing constant is that many important summary statistics can be expressed in terms of it.  However, there are of course important summaries that do not involve such simple representations, such as the median.  We were therefore unable to exploit our asymptotic series to improve on the current best approximation of Daly and Gaunt \cite{dg16}.  We leave this and related approximations as interesting further open problems.  

\appendix 

\section{Further proofs}\label{appendixaaa}

\subsection{Proof of Theorem \ref{thm3.1}}\label{appendixC}
The generalized hypergeometric function is defined by
\begin{equation*}{}_pF_q\left({a_1,\ldots,a_p\atop b_1,\ldots,b_q};z\right)=\sum_{k=0}^\infty\frac{(a_1)_n\cdots(a_p)_n}{(b_1)_n\cdots(b_q)_n}\frac{z^n}{n!}, 
\end{equation*}
provided none of the $b_j$ are nonpositive integers, where the Pochhammer symbol is $(a)_0=1$ and $(a)_n=a(a+1)(a+2)\cdots(a+n-1)$, $n\geq1$.  
In the case that $p\leq q$ the infinite series converges for all finite values of $z$ and defines an entire function. For more details see \href{http://dlmf.nist.gov/16.2}{\S 16.2} in \cite{NIST:DLMF}.
As noted by   Nadarajah \cite{n09}, $Z(\lambda,\nu)=$\! $_0F_{\nu-1}(;1,\ldots,1;\lambda)$ for integer $\nu$.  Thus, we can exploit the well-developed theory for asymptotics of the generalized hypergeometric function to obtained an asymptotic expansion for $Z(\lambda,\nu)$ when $\nu$ is an integer.  Indeed, \href{http://dlmf.nist.gov/16.11.E9}{16.11.9} in \cite{NIST:DLMF} can be used to write down the entire expansion for integer $\nu$.  However, this expansion is only valid for integer $\nu$, and, since $b_1=\cdots =b_{\nu-1}=1$, the coefficients of the lower order terms in the expansion must be obtained via a tedious limiting procedure.  We can, however, make use of a recent work on the asymptotics of the generalized hypergeometric function to prove Theorem \ref{thm3.1}. \\ 

\noindent{\bf{Proof of Theorem \ref{thm3.1}.}}  We obtain the expansion (\ref{thmex}) by appealing to results from the recent work of Lin and Wong \cite{lw16}, in which the large $z$ asymptotics of the generalized hypergeometric function 
${}_pF_q\left({a_1,\cdots,a_p\atop b_1,\cdots,b_q};z\right)$ are discussed. Here we will have $p=0$, all the $b_k=1$ and $q=\nu-1$.
In the proof of Lemma 4.3 in \cite{lw16} it is not essential that $q$ is a nonnegative integer and below we copy the main steps.

Lin and Wong \cite{lw16} used as their starting point the integral representation \href{http://dlmf.nist.gov/16.5.E1}{16.5.1} in \cite{NIST:DLMF} of the generalized hypergeometric function.  As mentioned in the introduction, the integral representation (\ref{mainintrep}) is just a simple generalization of \href{http://dlmf.nist.gov/16.5.E1}{16.5.1} in \cite{NIST:DLMF}.  We rewrite it as
\begin{equation}\label{thmintrep}
Z(\lambda,\nu)=\frac{1}{2\pi i}\int_L\frac{\Gamma(t+1)\Gamma(-t)\left(-\lambda\right)^t}{\left(\Gamma(t+1)\right)^\nu}\, dt,
\end{equation}
and use the inverse factorial expansion (again, see \cite{lw16}, or \S2.2.2 in Paris and Kaminski \cite{PK01}, and for more information about factorial series see Weniger \cite{Weniger})
\begin{equation}\label{thminvfact}
\left(\Gamma(t+1)\right)^{-\nu}=\frac{\nu^{\nu (t+1/2)}}{\left(2\pi\right)^{(\nu-1)/2}}\sum_{j=0}^\infty\frac{c_j}{\Gamma(\nu t+(1+\nu)/2+j)},
\end{equation}
in the integral in (\ref{thmintrep}) and combine this with the proof of Lemma 4.3 in \cite{lw16}. 
The result is asymptotic expansion (\ref{thmex}). 

To compute the $c_j$ we substitute the Stirling approximation 
\begin{equation*}\Gamma(x)\sim e^{-x}x^x\bigg(\frac{2\pi}{x}\bigg)^{1/2}\sum_{k=0}^\infty \frac{g_k}{x^k}, \quad x\rightarrow\infty,
\end{equation*}
(see  \href{http://dlmf.nist.gov/5.11.E3}{5.11.3} in \cite{NIST:DLMF})  into both sides of (\ref{thminvfact}).  Here $g_0=1$, $g_1=\frac{1}{12}$, $g_2=\frac{1}{288}$, and further values are given in \href{http://dlmf.nist.gov/5.11.E4}{5.11.4} -- \href{http://dlmf.nist.gov/5.11.E6}{5.11.6} in \cite{NIST:DLMF}.  In this way we obtain large $t$ asymptotic expansions for both sides of (\ref{thminvfact}) and we can compare the coefficients to find the $c_j$.
Computer algebra is very useful for this process; we used Maple to obtain the first eight coefficients. The result is:
\begin{align}\label{cccdef}
&c_0=1,\quad c_1=\frac{\nu^2-1}{24}, \quad c_2=\frac{\nu^2-1}{1152}\left(\nu^2+23\right),
\quad c_3=\frac{\nu^2-1}{414720}\left(5\nu^4-298\nu^2+11237\right),\nonumber \\
&c_4=\frac{\nu^2-1}{39813120}\left(5\nu^6-1887\nu^4-241041\nu^2+2482411\right),\nonumber \\
&c_5=\frac{\nu^2-1}{6688604160}\left(7\nu^8-7420\nu^6+1451274\nu^4-220083004\nu^2+1363929895\right),\nonumber \\
&c_6=\frac{\nu^2-1}{4815794995200}\left(35\nu^{10}-78295\nu^8+76299326\nu^6+25171388146\nu^4\right . \\
&\qquad\qquad\qquad\qquad\qquad \left.-915974552561\nu^2+4175309343349\right),\nonumber \\
&c_7=\frac{\nu^2-1}{115579079884800}\left(5\nu^{12}-20190\nu^{10}+45700491\nu^8-19956117988\nu^6\right .\nonumber \\
&\qquad\qquad\qquad\qquad\qquad \left.+7134232164555\nu^4-142838662997982\nu^2+525035501918789\right).\nonumber
\end{align}
\hfill $\Box$

\subsection{Proof of Lemma \ref{lem1}}\label{appendixA}

Let us first note two results that were derived using Stein's method.  Theorem \ref{thm1} is proved in Stein \cite{stein2}, whilst Theorem \ref{thm2}  is a special (and slightly simplified) case of the general bound of Theorem 3.5 of Gaunt \cite{gaunt normal}.  The simplified bound is, however, sufficiently tight for the purpose of proving part $(ii)$ of Lemma \ref{lem1}.

\begin{theorem}\label{thm1}Stein \cite{stein2}, (1986). Let $X_1,\ldots,X_n$ be i.i.d$.$ random variables with $\mathbb{E}X_1=0$, $\mathbb{E}X_1^2=1$ and $\mathbb{E}|X_1|^3<\infty$.  Set $W=\frac{1}{\sqrt{n}}\sum_{i=1}^nX_i$ and let $Z\sim N(0,1)$.  Suppose $h:\mathbb{R}\rightarrow\mathbb{R}$ has a bounded first derivative on $\mathbb{R}$.  Then
\begin{equation*}|\mathbb{E}[h(W)]-\mathbb{E}[h(Z)]|\leq\frac{\|h'\|_\infty}{\sqrt{n}}\big(2+\mathbb{E}|X_1|^3\big),
\end{equation*}
where $\|h'\|_\infty=\sup_{x\in\mathbb{R}}|h'(x)|$.
\end{theorem}

\begin{theorem}\label{thm2}Gaunt \cite{gaunt normal}, (2015). 
Let $X_1,\ldots,X_n$ and $W$ be defined as in Theorem \ref{thm1}, but with the additional assumption that $\mathbb{E}|X_1|^6<\infty$.  Suppose $h:\mathbb{R}\rightarrow\mathbb{R}$ is an even function and is twice differentiable with first and second derivative bounded on $\mathbb{R}$.  Then there exists a constant $C$ independent of $n$ such that
\begin{equation*}|\mathbb{E}[h(W)]-\mathbb{E}[h(Z)]|\leq\frac{C}{n}\big(\|h'\|_\infty+\|h''\|_\infty)\big(1+|\mathbb{E}X_1^3|\big)\mathbb{E}|X_1|^6.
\end{equation*}
\end{theorem}

\noindent{\bf{Proof of Lemma \ref{lem1}.}} Let $X_\alpha\sim\mathrm{Po}(\alpha)$ and set $\tilde{X_\alpha}=\frac{X_\alpha-\alpha}{\sqrt{\alpha}}$.  We shall suppose $\alpha\geq1$ (later we shall let $\alpha\rightarrow\infty$).  Also, let $Y_1,\ldots,Y_{\lfloor\alpha\rfloor}$  be i.i.d$.$ random variables following the $\mathrm{Po}\big(\frac{\alpha}{\lfloor\alpha\rfloor}\big)$ distribution, where the floor function $\lfloor x\rfloor$ is the greatest integer less than or equal to $x$.  Then, by a standard result, $\sum_{i=1}^{\lfloor\alpha\rfloor}Y_i\sim \mathrm{Po}(\alpha)$, and so is equal in distribution to $X_\alpha$.  Therefore $\tilde{X_\alpha}\stackrel{\mathcal{D}}{=}\frac{1}{\sqrt{\alpha}}\big(\sum_{i=1}^{\lfloor\alpha\rfloor}Y_i-\alpha\big)$.  In order to apply Theorems \ref{thm1} and \ref{thm2}, we note that
\begin{equation*}\tilde{X_\alpha}\stackrel{\mathcal{D}}{=}\frac{1}{\sqrt{\lfloor\alpha\rfloor}}\sum_{i=1}^{\lfloor\alpha\rfloor}\tilde{Y}_i,
\end{equation*} 
where
\begin{equation*}\tilde{Y_i}=\sqrt{\frac{\lfloor\alpha\rfloor}{\alpha}}\bigg(Y_i-\frac{\alpha}{\lfloor\alpha\rfloor}\bigg), \quad i=1,\ldots,\lfloor\alpha\rfloor.
\end{equation*}
The random variables $\tilde{Y}_1,\ldots,\tilde{Y}_n$ are i.i.d$.$ with $\mathbb{E}\tilde{Y}_1=0$ and $\mathbb{E}\tilde{Y}_1^2=1$.  The absolute moments of $\tilde{Y}_1$ up to sixth order are also finite and are $\mathcal{O}(1)$ as $\alpha\rightarrow\infty$.  Parts ($i$) and ($ii$) of the lemma (for which different assumptions are made on the function $h$) now follow from applying Theorems \ref{thm1} and \ref{thm2} to bound the quantity $|\mathbb{E}[h(\tilde{X_\alpha})]-\mathbb{E}[h(Z)]|$ and noticing that the resulting bounds are of order $\mathcal{O}(\alpha^{-1/2})$ and $\mathcal{O}(\alpha^{-1})$, respectively, as $\alpha\rightarrow\infty$. \hfill $\Box$

\subsection*{Acknowledgements}
RG is supported by a Dame Kathleen Ollerenshaw Research Fellowship.  SI is supported by a grant from the National Institute of Mental Health (5R01 MH060952-09). AOD is supported by a research grant (GRANT 11863412/70NANB15H221) from the National Institute of Standards and Technology.  The authors would like to thank the referees for their helpful comments and suggestions.

\end{document}